\documentclass{article}
\usepackage{fixltx2e}   
\usepackage[T1]{fontenc}
\usepackage[utf8]{inputenc}
\usepackage[english,british]{babel}
\usepackage[final]{microtype}  

\usepackage{amsmath,amsthm,amssymb, url}
\usepackage{tikz}   
\usepackage[matrix,arrow]{xy}    
\usepackage{ifpdf}

\newtheorem{theorem}{Theorem}[subsection]
\newtheorem{remark}[theorem]{Remark}
\newtheorem{proposition}[theorem]{Proposition}

\newtheorem{definition}[theorem]{Definition}
\newtheorem{lemma}[theorem]{Lemma}

\ifpdf
\usepackage[final,pdftex,unicode]{hyperref}   
        \hypersetup{%
                bookmarksnumbered=true, 
            bookmarksopen=true,
            pageanchor=true,
            naturalnames=true
                }
\fi

\newcommand{\Spin}{\mathrm{Spin}}
\newcommand{\EndC}{\mathrm{End}}
\newcommand{\Cliff}{C\ell}
\newcommand{\Herm}{\mathrm{Herm}}
\newcommand{\F}{\mathrm{F}_4}
\renewcommand{\P}{\mathrm{P}_4}

\newenvironment{smatrix}{\left(\begin{smallmatrix}}{\end{smallmatrix}\right)}

\newcommand{\octo}{\mathbb{O}}
\newcommand{\R}{\mathbb{R}}

\ifpdf
   \renewcommand{\C}{\mathbb{C}}
\else
   \newcommand{\C}{\mathbb{C}}
\fi
\newcommand{\V}{\mathbb{V}_9}

\newcommand{\Id}{\mathit{Id}} 
\newcommand{\Tr}{\mathrm{Tr\,}}
\newcommand{\plane}{\mathbb{OP}^2_0}
\newcommand{\affOP}{\widehat{\mathbb{OP}}{}^2_0}

\title{Hyperplane section $\plane$ of the complex Cayley plane as the homogeneous space $\mathrm{F_4/P_4}$}
\author{Karel Pazourek, V\'it Tu\v{c}ek\footnote{The author was supported by GA\v{C}R 201/09/H012 and by SVV 301-09/10822.}, Peter Franek\footnote{The author was supported by MSM 0021620839 and GA\v CR 201/08/397.}}
\begin{document}
   \maketitle
   \begin{abstract}
      We prove that the exceptional complex Lie group $\F$ has a transitive action on the hyperplane section of the complex Cayley plane $\mathbb{OP}^2$. Our proof is direct and constructive.  We use an explicit realization of the vector and spin actions of $\Spin(9,\C) \leq \F$. Moreover, we identify the stabilizer of the $\F$-action as a parabolic subgroup $\P$ (with Levi factor $\mathrm{B_3T_1}$) of the complex Lie group $\F$. In the real case we obtain an analogous realization of $\F^{(-20)}/\P$.
   \end{abstract}

\section{Introduction}

The octonionic projective plane $\mathbb{OP}^2$, also called Cayley plane or octave plane, has been thoroughly treated in the literature. It appears in numerous contexts. It is a projective plane where the Desargues axiom does not hold. It was firstly considered by Ruth Moufang \cite{Moufang}, who found a relation of the so called little Desargues axiom and the alternativity of the coordinate ring. It is well known that $\mathbb{OP}^2$ is a Riemanian symmetric manifold $\mathrm{F}_4/\mathrm{Spin}(9)$. Due to its relation to the exceptional Jordan algebra $\mathcal{J}_3(\octo)$, there is also a connection of this plane to a model of quantum mechanics considered by Neumann, Jordan and Wigner \cite{JNW}. More recently, authors of \cite{Dray} show that the the Cayley plane consists of normalized solutions of a Dirac equation.  For more details and connections with physics we refer to the article by Baez \cite{Baez}.

It is possible to mimic the the construction of $\mathbb{OP}^2$ via equivalence classes of triples of octionios, but usually Freudenthal's approach via the exceptional Jordan algebra $\mathcal{J}_3(\octo)$ is used. The idea is that lines in space correspond to projectors with one-dimensional image. Hence the Cayley plane can be defined as elements of (real) projectivization of $\mathcal{J}_3(\octo)$ of rank one. Now the rank for octonionic matrices is a bit tricky due to nonassociativity and requires the definition of Jordan cross product of these matrices. For details we refer to Jacobson's monograph \cite{Jacobson}. There one can also find a classification of orbits of the automorphism group of $\mathcal{J}_3(\octo)$ (which is $\F$) from which it follows that $\mathbb{OP}^2$ is a homogeneous space. (The isotropy subgroup is determined in for example in \cite{Harvey,Springer}.)  In fact, Jacobson's book \cite{Jacobson} treats octonionic algebras over  general field and hence we get the definition of the complex Cayley plane as well. This space is also of geometric interest, as it is an exceptional member of the Severi varieties -- the unique extremal varieties for secant defects. For details see \cite{Landsberg, Landsberg2}.

Now, let us consider the intersection of the complex Cayley plane with the hyperplane given by traceless matrices $\mathcal{J}_0 := \{ A \in \mathcal{J}_3(\mathbb{O_C})\, |\, \Tr A = 0 \}$. The resulting space is studied in \cite{Landsberg,Landsberg2}, where the authors call it the generic hyperplane section and denote it by $\plane$. It is a total space of a twistor fibration over the real Cayley plane (see \cite{Atiyah, Friedrich}). Because $\plane$ is a complex projective variety, the stabilizer is a parabolic subgroup $P$ of $\F$. It is well known that the Cartan geometry modeled on the pair $(\F,\P)$ is rigid, i.e. any regular normal Cartan geometry of this type is locally isomorphic to the homogeneous model. The real version of this pair corresponding to the group $\F^{(-20)}$ appears as a conformal infinity of the Einstein space $\mathbb{O}H^2$ \cite{Biquard}. The geometry obtained is called `octonionic-contact', because there is a naturally defined eight-dimensional maximally nonintegrable subbundle of the tangent bundle. The contact geometry in the classical sense (studied for example in \cite{K_clas1, K_clasp}) is also present among the homogeneous spaces of the group $\F$ -- namely the one whose isotropy group is the parabolic subgroup corresponding to the other `outer' simple root of the Lie algebra of $\mathfrak{f}_4$. 

The authors of \cite{Landsberg} state that the isomorphism $\plane = \F/\P$ is suggested by `geometric folding'. A rigorous proof of this isomorphism can be gleaned from \cite{Jacobson}. This proof however requires a lot of the theory of nonassociative algebras, most notably it needs Jordan coordanization theorem. Quite short proof can be given using the Borel fixed point theorem. In a hope to make $\plane$ more accesible to geometrically inclined audience, we present a \emph{constructive} proof of the transitivity of the action of $\F$ on $\plane$ based on the representation theory of complex spin groups. From the theory of nonassociative algebras only Artin's theorem is needed. Following the approach of \cite{Harvey}, we explicitly realize the spin groups $\Spin(9,\C)$ and $\Spin(8,\C)$ as subgroups of  $\F$ and we use the description of their actions to find the reduction of an arbitrary element to a previously chosen one. 

 After some necessary definitions in section two, we describe explicitly the presentations of $\Spin(9,\C)$ and $\Spin(8,\C)$ inside of $\EndC(\octo^2)\otimes_\R\C$ in section three. We also explicitly describe vector and spinor representations of $\Spin(9,\C)$ in such a way that their image is inside $\F$. Section three continues with the proof of the transitivity of the action of $\F$ on $\plane$. We conclude by dealing with the real case. In the last section we compute the stabilizer of a point.

We are thankful to Mark MacDonald who pointed out Jacobson's work. Also, the role of Svatopluk Kr\'ysl was indispensable.

\section{Notations and definitions}
   \subsection{Complexified octonions and the hyperplane section}
   For a comprehensive reference on octonionic algebras over any field we refer to \cite{Springer}.
   We denote by $\octo$ the octonionic algebra over the field of complex numbers. The complex-valued `norm' on $\octo$ is denoted by $N$. The algebra $\octo$ is normed ($N(ab)=N(a)N(b)$) but it fails to be a division ring, since $N$ is isotropic. This algebra is not associative. Nevertheless, it is alternative, which means that the trilinear form (called the associator) $ [u,v,w] \mapsto (uv)w-u(vw) $ is completely skew-symmetric. Later on we will use the so called \emph{Artin's theorem} which states that any subalgebra of an alternative algebra generated by two elements is associative. It follows that products involving only two variables can be written without parenthesis unambiguously.

   The symbol $L_u$ denotes the operator of left multiplication by $u$, i.e. $L_u (v) := uv$ for any $v\in\octo$. Note that $L_uL_v \neq L_{uv}$ in general due to the nonassociativity of octonionic algebras. 

   Since there is up to isomorphism only one octonionic algebra over $\C$ we can think of $\octo$ in the following way: $\octo = \mathbb{O_R\otimes C} = \mathbb{O_R\otimes_R C}$, where $\mathbb{O_R}$ is the classical real algebra of octonions (\cite{Baez}). The multiplication on this tensor product is canonically defined by
   \[
        (o_1\otimes z_1)(o_2\otimes z_2):=o_1o_2\otimes z_1z_2 \text{ for } o_1,\,o_2\in\mathbb{O},z_1,z_2\in\C
   \]
   and conjugation is given by $\overline{o\otimes z}:=\bar{o}\otimes z$.

   The multiplication of an arbitrary element $o\otimes z\in\octo$ by a complex number $w$ is understood in the sense of multiplication by element $1\otimes w$, i.e. $w(o\otimes z):=o\otimes(wz)$. We identify the elements of $\mathbb{R\otimes C}$ with complex numbers under the canonical isomorphism $r\otimes w \mapsto rw$, for $r\in\R$, $w\in \C$. The real and imaginary parts of $o\otimes z$ are defined to be $(\Re\, o )\otimes z$ and $(\Im\, o)\otimes z$ respectively.

   The complex valued quadratic form $N$ is given by
    \[
      N(o\otimes z):=o\bar{o}zz,\qquad o\in\mathbb{O},z\in\C.
   \]
   Following Springer \cite{Springer}, we denote by $\langle \cdot, \cdot \rangle$  the double of the bilinear form associated to $N$, $\langle x,y\rangle = N(x+y)-N(x)-N(y)$. An octonion $u\in\octo$ is pure imaginary if and only if $\langle u, 1\rangle =0$.

   For later use, we will record here several useful identities which holds in any octonionic algebra and whose proof can also be found in \cite{Springer}
   \begin{subequations}
   \begin{gather}
      \langle xy, z\rangle = \langle y,\bar{x}z\rangle\nonumber\\
      x(\bar{x}y) = N(x)y \label{eq:norma}\\  
      u(\bar{x}y) + x(\bar{u}y) = \langle u,x\rangle y  \label{eq:skal_souc}\\ 
      u(\bar{x}(uy)) = ((u\bar{x})u)y. \label{eq:mouf} 
   \end{gather}
\end{subequations}
   Due to nonassociativity of the algebras involved we need to make clear distinction between associative algebras of $\C$-linear endomorphisms, which we denote by~$\mathrm{End}$, and the possibly nonassociative algebras of $n\times n$ matrices with entries in some algebra $\mathbb{F}$ which are denoted by $M(n,\mathbb{F})$.

   The conjugation on $\octo$ naturally defines the conjugation on $M(n,\octo)$. The conjugate of an element $A \in M(n,\octo)$ is denoted by $\bar{A}$. The symbol $\Herm(n,\octo)$ stands for the set of $n \times n$ hermitian matrices over $\octo$, i.e.
   \[
      \Herm(n,\octo) = \{A \in M(n,\octo) | \bar{A}^T = A\}.
   \]
  We denote the subspace of trace-free matrices by lower index $\Herm_0(n,\octo)$. All tensor products in this article are taken over the real numbers.

   \emph{The complex exceptional Jordan algebra $\mathcal{J}_3(\octo)$} is the vector space $\Herm(3,\octo)$ endowed with the symmetric product $\circ:\Herm(3,\octo)\times \Herm(3,\octo)\to \Herm(3,\octo)$ defined by $A\circ B:= \frac{1}{2}(AB+BA)$. 

  Now we define the basic object of our interest.
  \begin{definition}
    The hyperplane section of the complex Cayley plane $\plane$ is the projectivization over $\C$ of the following subset of $\mathcal{J}_3(\octo)$
   \[
      \affOP:=\left\{ A\in \Herm(3,\,\octo) \middle|\; A^2=0,\; {\rm tr}A=0,\; A\neq 0\right\}.
   \]
  \end{definition}


   \subsection{The spin groups}
   For an $n$-dimensional complex vector space $\mathbb{V}$ and a nondegenerate quadratic form $N$ on $\mathbb{V}$, we denote the corresponding Clifford algebra by $\Cliff(\mathbb{V},N)$ (our convention is $vv=-N(v)$).
   The spin group of $\Cliff(\mathbb{V},N)$ is denoted by $\Spin(\mathbb{V},N)$. It is generated inside $\Cliff(\mathbb{V},N)$ by products $uv$, $u,v\in \mathbb{V}$ where $N(u)=N(v)=1$. By $\Spin(n,\C)$ we denote the spin group associated to the standard quadratic form $\sum_{i=1}^n z_i^2$ on $\C^n$.

   For $w\in\C$ we define \emph{the generalized complex sphere}
   \[
      S^{n-1}(w) = \{ 0\neq z \in \mathbb{V}\, |\, N(z)=w^2 \}.
   \]
As a consequence of Witt's theorem we have
   \begin{lemma}\label{lem:vect_trans}
      The group $\Spin(n,\C)$ acts transitively via the vector representation on the generalized complex spheres.
   \end{lemma}
%
   \subsection[Lie algebra]{Complex Lie algebra $\mathfrak{f}_4$}
\label{liealgebra}
   The complex exceptional Lie group $\F$ can be defined as the automorphism group $\mathsf{Aut}(\mathcal J)$ of the complex exceptional Jordan algebra $(\mathcal{J}_3(\octo),\circ)$ (see \cite{Springer}). In other words $\F$ is the subgroup of $\mathrm{GL}(27,\C)$ such that $g\in\F$ if and only if $g(A\circ B) = gA \circ gB$ for every $A,B \in \Herm(3,\octo)$.

  The action of $\F$ preserves the trace on $\Herm(3,\octo)$. This can be easily proved using the  equality
      \[
       \Tr A = \frac{1}{3} \Tr (B \mapsto A\circ B).
      \]

      It is easy to verify that the action of $\mathrm{O}(3,\C)$ on $\Herm(3,\octo)$ given by
      \begin{equation*}\label{eq:so3}
         \mathrm{O}(3,\,\C)\ni g \longmapsto (A \mapsto gAg^T), \quad A\in\Herm(3,\octo)
      \end{equation*}
      defines an injective group homomorphism $\mathrm{O}(3,\C) \hookrightarrow \F$.

   Now we present basic facts about the complex simple Lie algebra $\mathfrak{f}_4$ of the group $\F$. We shall use these facts as well as the properties of the root system of the Lie algebra $\mathfrak{f}_4$ in the last section of this text. Details can be found in \cite{Bourbaki}.

There exist a choice of the Cartan subalgebra $\mathfrak{h}$ of $\mathfrak{f}_4$, an orthonormal 
(with respect to the Killing form of $\mathfrak{f}_4$) basis $\{\epsilon_i\}_{i=1}^4$ of $\mathfrak{h}^*$ 
and a choice of simple roots 
\[\Delta=\left\{\alpha_1=\epsilon_2-\epsilon_3,\;\alpha_2=\epsilon_3-\epsilon_4,\;\alpha_3=\epsilon_4,\;\alpha_4=\frac{1}{2}(\epsilon_1-\epsilon_2-\epsilon_3-\epsilon_4)\right\}.\]

In this convention the Dynkin diagram is
   
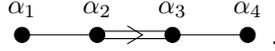
\begin{center}
     \begin{tikzpicture}[koren/.style={circle,draw,fill=black,inner sep=0pt,minimum size=2mm}]
        \draw (0,0) -- (1,0);
        \draw (1,-.05) -- (2,-.05);
        \draw (1,.05) -- (2,.05);
        \draw (2,0) -- (3,0);
        \draw (1.4,.12) -- (1.6,0) -- (1.4,-.12);
        \node at (0,0) [koren,label=above:$\alpha_1$] {};
        \node at (1,0) [koren,label=above:$\alpha_2$] {};
        \node at (2,0) [koren,label=above:$\alpha_3$] {};
        \node at (3,0) [koren,label=above:$\alpha_4$] {};
     \end{tikzpicture}.
  \end{center}

The set $\Delta$ determines the set of positive roots $\Phi^+$ . 
%
For any root $\alpha$, we define the coroot $H_\alpha\in\mathfrak{h}$ by 
$\lambda(H_{\alpha})=2\langle \lambda, \alpha\rangle/2\langle \alpha, \alpha \rangle$, where $\langle\;,\;\rangle$ is the Killing form.

   The fundamental weights $\{\varpi_i\}_{i=1}^4$  are defined as the dual basis to the simple coroots.
%
%
%
%
   We denote the irreducible representation of $\mathfrak{f}_4$ with the highest weight $\lambda$ by $\varrho_{\lambda}$. 

\section[Action]{Action of $\F$ on $\affOP$}
   In this section we explicitely describe the group $\Spin(9,\C)$ as a multiplicative subgroup of $\EndC(\octo^2)\otimes\C$ and construct its representation on $\Herm(3,\octo)$. Using this representation, we prove that $\F$ acts transitively on the hyperplane section $\affOP$. 

   \subsection[Spin groups]{Realisation  of $\Spin(9,\C)$}
   First we need an auxiliary result concerning the Clifford algebra $\Cliff(\octo,N)$.
   \begin{lemma}\label{lem:clif8}
   
    The map $\mu: \octo \to \EndC(\octo^2)$ given by
      \[
            u   \longmapsto \begin{pmatrix}
                                    0 & L_u \\
                                    -L_{\bar{u}} & 0
                              \end{pmatrix}
      \]
      can be uniquely extended to the isomorphism of complex associative algebras $\Cliff(\octo,N)\simeq \EndC \left(\octo^2\right)$.       
   \end{lemma}
   \begin{proof}
   	Easy calculation and \eqref{eq:norma} shows that $\mu(u)\mu(u) = -N(u)\Id$. Using the universal property of Clifford algebras and the fact that the algebra  $\Cliff(8,\C)$ is simple (see \cite{GoodmanWallach}), we immediately get the result.
   \end{proof}

   Let $\mathbb{V}_9$ be the complex vector space $\C\oplus\octo$. We define the quadratic form $N'$ by $(r,u)\mapsto r^2+N(u)$. Let $\kappa: \V \to \EndC(\octo^2)\otimes\C$ be the homomorphism\footnote{The scalar multiplication on this complex algebra acts only on the first part of the tensor product, i.e. $w\cdot (A\otimes z) = (wA)\otimes z$ for $w,z\in \C$, $A\in \EndC(\octo^2)$.} given by
   \[
      \kappa:\,(r,u) \longmapsto \begin{pmatrix}
                           r & L_u \\
                           L_{\bar{u}} & -r
                        \end{pmatrix}\otimes\imath.
   \]

   \begin{proposition}	
    The Clifford algebra $\Cliff(\V, N')$ is isomorphic (as an associative algebra) to $\EndC(\octo^2)\otimes\C$.
   \end{proposition}
   \begin{proof}
      It is known (see e.g. \cite{GoodmanWallach}) that $\Cliff(\V,N')\simeq M(16,\C)\oplus M(16,\C)$.
      Calculation and \eqref{eq:norma} shows that $\kappa(r,u)\kappa(r,u)= -N'(r,u) \Id$. The universal mapping property of Clifford algebras gives us the following commutative diagram
      \begin{center}
    \ \hbox{
    \xymatrix{
        \V \ar[dr]_(.4){\kappa} \ar[r]^(.28){i} & M(16,\,\C)\oplus M(16,\C)\ar[d]^(.45){f} \\
    & \EndC(\octo^2)\otimes\C\ .
    }}
      \end{center}
    Because  $\kappa(-1,0)\kappa(0,u) = \mu(u)\otimes 1$, we see that the image of $f$ generates the subalgebra $\EndC(\octo^2)\otimes 1$.  The equality 
\[
 \begin{pmatrix} A &B\\C &D\end{pmatrix}\otimes \imath =\begin{pmatrix} 1 & 0 \\ 0 &-1\end{pmatrix}\otimes \imath \cdot \begin{pmatrix} A & B\\-C & -D \end{pmatrix} \otimes 1
\]
implies that the image of $f$ generates the whole algebra $\EndC(\octo^2)\otimes \C$. Since the dimensions of the considered alegbras are the same, it follows that $f$ is an isomorphism.
%
   \end{proof}

   \begin{lemma}
      The spin group $\Spin(\V,N')$ is generated (inside $\EndC(\octo^2)\otimes\C$) by elements of the form
      \begin{equation*}\label{eq:spin9gen}
         g_{r,u} := \begin{pmatrix}
            r & -L_u \\
            L_{\bar{u}} & r
         \end{pmatrix}\otimes 1, \quad r \in \mathbb{C},\ u \in \octo, \quad r^2 + u\bar{u} = 1
      \end{equation*}
   \end{lemma}
   \begin{proof}
      The spin group is by definition generated by products of the form\linebreak[4] $\kappa(r,u)\kappa(s,v)$, where $N'(r,u)=N'(s,v)=1$. Since $g_{r,u} = \kappa(r,u)\kappa(-1,0)$ and $\kappa(r,u)\kappa(s,v) = g_{r,u}g_{-s,v}$, the lemma follows.
   \end{proof}

  For brevity we will identify $A\otimes 1 \in \EndC(\octo^2)\otimes \C$  with $A\in\EndC(\octo^2)$ from now on; i.e. $g_{r,u}=\begin{smatrix} r & -L_u \\
            L_{\bar{u}} & r\end{smatrix}$.

\subsection[Representations of spin 9]{Representations of $\Spin(\V,N')$}

   We will use the following decomposition of $\Herm(3,\octo)$
   \begin{equation*}
         \begin{pmatrix}
         r_1 & \bar{x}_1 & \bar{x}_2 \\
         x_1 & r_2 & x_3 \\
         x_2 & \bar{x}_3 & r_3
      \end{pmatrix} = \begin{pmatrix}r_1 &0&0\\0&0&0\\0&0&0\end{pmatrix} +
      \begin{pmatrix}
         0 & \bar{x}_1 &\bar{x}_2\\x_1&0&0\\x_2&0&0
      \end{pmatrix} +
      \begin{pmatrix}
         0&0&0\\
         0&s&x_3\\
         0&\bar{x}_3&-s
      \end{pmatrix}+
      \begin{pmatrix}
       0 &0 &0\\
       0 &t&0\\
       0&0&t
      \end{pmatrix}
   \end{equation*}
   in order to define the action of $\Spin(\V,N')$ on it. In other words -- we take the $\C$-linear isomorphism $\Herm(3,\octo)\to 
                        \C\oplus\octo^2\oplus\Herm_0(2,\octo)\oplus\C$ and we endow each of the spaces in the decomposition with an action of $\Spin(\V,N')$. The $\octo^2$ summand will be \emph{the spinor part} and we will call the $\Herm(2,\octo)_0$ summand \emph{the vector part}.

  \begin{lemma}\label{lem:rep}
  Let $\Phi$ be the linear isomorphism between the space of trace-free hermitian matrices $\Herm_0(2,\octo)$ and $\kappa(\V)$ defined by
  \[
  \Phi: \begin{pmatrix}
                    s & x \\
               \bar{x} & -s
           \end{pmatrix} \mapsto \begin{pmatrix}
                                        s & L_x \\
                                     L_{\bar{x}} & -s
                                 \end{pmatrix}\otimes\imath
  \]
  and let $\varrho_V$ be the vector representation of $\Spin(\V,N')$. 

  If we define the representation of $\Spin(\V,N')$ on $\Herm_0(2,\octo)$ by \(\xi_V(g)a := \Phi^{-1}\bigl(\varrho_V(g)\Phi(a) \bigr)\),
  the following formula holds for the generators $g_{r,u}$ of $\Spin(\V,N')$
  \begin{align}
    \xi_V(g_{r,u}) \begin{pmatrix}
                    s & x \\
               \bar{x} & -s
           \end{pmatrix} & = \left[ \begin{pmatrix} r & -u \\ \bar{u} & r\end{pmatrix}
				  \begin{pmatrix}
					    s & x \\
				      \bar{x} & -s
				  \end{pmatrix}\right]
				  \begin{pmatrix} r & u \\ -\bar{u} & r\end{pmatrix} \nonumber\\
			& = \begin{pmatrix}
			      s\bigl(r^2-N(u)\bigr)-r\langle x,u \rangle & 2rsu + r^2x - u\bar{x}u \\
            2rs\bar{u} + r^2\bar{x} - \bar{u}x\bar{u} & -s\bigl(r^2-N(u)\bigr)+r\langle \bar{x},\bar{u} \rangle
			    \end{pmatrix}. \label{eq:spin9vect}
  \end{align}
  \end{lemma}
  \begin{proof}
      The vector representation of $\Spin(\V,N')$ is given by $v \mapsto gvg^{-1}$ where $v$ is an element of $\kappa(\V)$ and $g \in \Spin(\V,N')$. For $g_{r,u} = \kappa(r,u)\kappa(-1,0)$ we get $g_{r,u}^{-1} =  g_{r,-u}$.

      Thus we have the following formula for $\rho_V(g_{r,u})$  evaluated on
      $v=\begin{smatrix} s &L_x\\ L_{\bar{x}} & -s \end{smatrix} \otimes\imath$
      \[
         \begin{pmatrix}
            s\bigl(r^2-N(u)\bigr)-r(L_{u}L_{\bar{x}}+L_{x}L_{\bar{u}}) & 2rsL_u + r^2L_x - L_uL_{\bar{x}}L_u \\
            2rsL_{\bar{u}} + r^2L_{\bar{x}} - L_{\bar{u}}L_xL_{\bar{u}} & -s\bigl(r^2-N(u)\bigr)+r(L_{\bar{u}}L_{x}+L_{\bar{x}}L_{u})
         \end{pmatrix}\otimes \imath.
      \]
     From \eqref{eq:skal_souc} we have $L_{u}L_{\bar{x}}+L_{x}L_{\bar{u}} = 2L_{\langle x,u\rangle}$. With the help of the first Moufang identity \eqref{eq:mouf} we may substitute $L_{u}L_{\bar{x}}L_{u} = L_{(u\bar{x})u}$. Applying the isomorphism $\Phi$ to the result gives the expression for $\xi_V(g_{r,u})\Phi^{-1}(v)$ which agrees with \eqref{eq:spin9vect}. 
   \end{proof}

   The spinor representation of $\Spin(\V,N')$ acts on $\octo^2$ by (see chapter 6 of \cite{GoodmanWallach} for details)
   \[
      \xi_S(g_{r,u})(x_1,x_2) = \begin{pmatrix} r&-L_u\\L_{\bar{u}}&r\end{pmatrix}\begin{pmatrix}x_1\\x_2\end{pmatrix}=
      \begin{pmatrix}
       rx_1-ux_2\\\bar{u}x_1+rx_2
      \end{pmatrix}.
   \]


   We let the $\Spin(\V,N')$ act on the rest of the summands of $\Herm(3,\octo)$ trivially and denote the resulting action by $\xi$. 

   \begin{proposition}
      The representation $\xi$ is faithfull and preserves the Jordan product. In other words $\Spin(\V,N') \simeq \mathrm{Im}(\xi)$ is a subgroup of $\F$.
   \end{proposition}
   \begin{proof}
      Since the spinor representation $\xi_S$ is faithfull, the representation $\xi$ is faithfull as well. In order to prove that this action preserves the Jordan product we introduce the following three by three hermitian matrix
    \[
     G_{r,u} = \begin{pmatrix}
          1 & 0 & 0\\
          0 & r & -u \\
          0 & \bar{u} & r
         \end{pmatrix}\in\Herm(3,\octo), 
    \]
    where $(r,u) \in \V$ is of unit norm. Straightforward calculations reveal that $G_{r,u}^{-1} = G_{r,-u}$ and that  $G_{r,u}AG_{r,u}^{-1}$ gives the expression for the action of $\xi(g_{r,u})$ on $A$. Moreover the expression \(G_{r,u}AG_{r,u}^{-1}\) is unambiguous for any $A\in \Herm(3,\octo)$.

    Put $g=g_{r,u}$, $G=G_{r,u}$ for simplicity. For each  $A\in\Herm(3,\octo)$ we have
      \[
         (\xi(g)A)(\xi(g)A)=(GAG^{-1})(GAG^{-1}).
      \] Let's suppose for a moment that $(GAG^{-1})(GAG^{-1}) = G(A(G^{-1}G)A)G^{-1}$. Then we would have 
      \begin{align}\label{eq:ga2}
            (\xi(g)A)(\xi(g)A)=\xi(g)(A^2)
      \end{align} for any $A\in \Herm(3,\octo)$.
    Using this equality for $A+B$ instead of $A$ we would get on the left hand side
      \begin{multline*}
         \left(\xi(g)(A+B)\right)\left(\xi(g)(A+B)\right)  = \big(\xi(g)A+\xi(g)B\big)\big(\xi(g)A+\xi(g)B\big)=\\
             = (\xi(g)A)^2 +(\xi(g)A)(\xi(g)B)+(\xi(g)B)(\xi(g)A)+(\xi(g)B)^2,
      \end{multline*}
      while the right hand side would equal
      \begin{align*}
       \xi(g)\left((A+B)^2\right) & = \xi(g)(A^2)+\xi(g)(AB)+\xi(g)(BA)+\xi(g)(B^2).
      \end{align*}
      Using \eqref{eq:ga2}  for $\xi(g)(A^2)$ and $\xi(g)(B^2)$ we would get that
      \[
         (\xi(g)A)(\xi(g)B)+(\xi(g)B)(\xi(g)A)= \xi(g)(AB+BA).
      \]
    
      Hence we only need to prove that we can rearrange the brackets in the expression $(GAG^{-1})(GAG^{-1}) $. From the  Artin's theorem follows that \[(u_1au_2)(u_3au_4) = u_1(a(u_2u_3)a)u_4,\] where $u_i$ are elements of the linear span of $\{r,u,\bar{u}\}$ and $a\in \octo$ is arbitrary. Using the same trick as above and writing this equality for $a+b$ instead of $a$ we get
      \[
       (u_1au_2)(u_3bu_4)+(u_1bu_2)(u_3au_4) = u_1(a(u_2u_3)b)u_4+u_1(b(u_2u_3)a)u_4.
      \] 
The equation
$$
((GAG^{-1})(GAG^{-1}))_{uv}=$$
$$
=\frac{1}{2}\sum_{i,j,\ldots ,m} (G_{u,i}A_{i,j}G^{-1}_{j,k})(G_{k,l}A_{l,m}G^{-1}_{m,v})+
(G_{u,l}A_{l,m}G^{-1}_{m,k})(G_{k,i}A_{i,j}G^{-1}_{j,v})
$$
and the fact that $G_{i,j}$ are from the linear span of $\{r,u,\bar{u}\}$ imply
$$(GAG^{-1})(GAG^{-1})=G(A(G^{-1}G)A)G^{-1}=GA^2 G^{-1}.$$
   \end{proof}

\begin{remark}
One could define the representation $\xi$ directly using the matrix $G_{r,u}$. It is however not clear that the expression $G_{r,u}AG^{-1}_{r,u}$ defines a representation due to the nonassociativity of the product of $\Herm(3,\octo)$.
\end{remark}

\subsection[The subgroup Spin(8,C)]{The subgroup $\Spin(8,\C)$}\label{ss:spin8}
	The usual  presentation of spin groups gives (see lemma \ref{lem:clif8}) the following set of generators of $\Spin(\octo,N)$
    \begin{equation*}\label{eq:spin8}
            \left\{
            \begin{pmatrix}
                L_uL_{\bar{v}} & 0 \\
                0 & L_{\bar{u}}L_v
            \end{pmatrix} \Big|\; u,v\in\octo,\; N(u)=N(v)=1
            \right\}.
    \end{equation*}
      One can obtain matrices of this form as products $\kappa(0,u)\kappa(0,-v)$
       which means that these generators are in fact elements of $\Spin(\V,N')$. The formula for the restriction of $\xi_V$ to the subgroup $\Spin(\octo,N)$ 
      \begin{equation}\label{eq:chi8}
         \xi_V\left(\begin{pmatrix}
                            L_uL_{\bar{v}} & 0 \\
                            0 & L_{\bar{u}}L_v
                    \end{pmatrix}\right)
                    \begin{pmatrix}
                        s& x_3 \\
                        \bar{x}_3 & -s
                    \end{pmatrix}
        = \begin{pmatrix}
                s & -u(\bar{v}x_3\bar{v})u \\
                -\bar{u}(v\bar{x}_3v)\bar{u} & -s
         \end{pmatrix}
        \end{equation}
	is easily proved using \eqref{eq:spin9vect}.

        Analogously, the action of $\Spin(\octo,N)$ on $\octo^2$ is given by
        \begin{equation}
	\label{eq:chi9}
         \xi_S\left(\begin{pmatrix}
                            L_uL_{\bar{v}} & 0 \\
                            0 & L_{\bar{u}}L_v
                    \end{pmatrix}\right)\begin{pmatrix}x_1\\x_2\end{pmatrix} =
                    \begin{pmatrix}
                     u(\bar{v}x_1)\\
                     \bar{u}(vx_2)
                    \end{pmatrix},
        \end{equation}
         which is the direct sum of  two inequivalent spinor representations of $\Spin(\octo,N)$. Please note that the quadratic form $N$ is invariant with respect to all the three inequivalent actions of $\Spin(\octo,N)$ on the vector space $\octo$.

   \subsection[Transitivity of the action]{Transitivity of the $\F$ action on $\affOP$}
   \begin{lemma}
\label{fundlema}
    Let \[
        A = \begin{pmatrix}
            -2t & \bar{x}_1 & \bar{x}_2 \\
            x_1 & t + s & \bar{x}_3 \\
            x_2 & x_3 & t-s
            \end{pmatrix}
    \]  be an element of $\affOP$. Then the vector part of $A$ is isotropic (i.e. $s^2 + N(x_3) = 0$) if and only if $N(x_1)=N(x_2)=0$ and if and only if $t=0$.
   \end{lemma}
   \begin{proof}
    The statement is a straightforward consequence of the fact that diagonal elements 
of $A^2$ must equal zero.
%
%
%
   \end{proof}

   \begin{theorem}
      The group $\F$ acts transitively on $\affOP$. For every $A\in \affOP$ there exists  $g\in\F$ such that 
      \begin{equation}\label{eq:canform}
	g\cdot A=
         \begin{pmatrix}
                     \imath & 1 & 0 \\
                     1 &-\imath & 0 \\
                     0 & 0 & 0
                  \end{pmatrix}.
      \end{equation}
   \end{theorem}
   \begin{proof}
  First we suppose that $A\in\affOP$ has nonisotropic vector part. In such case we can use the lemma~\ref{lem:vect_trans} to prove that there exists an element $h_1\in \Spin(\V, N')$ such that
        \[
        \xi(h_1)A=\begin{pmatrix}
                           r_1 & \bar{x}_1 & \bar{x}_2 \\
                           x_1 & r_2 & 0 \\
                           x_2 & 0 & r_3
                  \end{pmatrix} \text{, with } r_1, r_2, r_3\in\C,\; x_1,x_2\in\octo.
        \]
Let as denote $\xi(h_1)=:g_1\in\F$.
The matrix $(g_1\cdot A)^2$ has the form
\begin{equation}
\begin{pmatrix}
\label{matrixsquare}
r_1^2+N(x_1)+N(x_2) & \bar{x}_1(r_1+r_2) & \bar{x}_2 (r_1+r_3)\\
x_1(r_1+r_2) & r_2^2 + N(x_2) & x_1 \bar{x}_2 \\
x_2(r_1+r_3) & x_2\bar{x}_1 & r_3^2+N(x_2)
\end{pmatrix}.
\end{equation}
This is a zero matrix, in particular $N(x_1)N(x_2)=N(x_1 \bar{x}_2)=0$, so $x_1$ and $x_2$ can not be both non-isotropic. On the other hand, they can not be both isotropic because of the lemma \ref{fundlema}.

Assume first that $N(x_1)\neq 0$ and $N(x_2)=0$.
The action of $\Spin(\octo, N)$ preserves the vector part $\begin{smatrix}r_2 & 0\\ 0 & r_3\end{smatrix}$
of $g_1\cdot A$ because of $(\ref{eq:chi8})$. 
Let $h_2:=\kappa(0,-1)\kappa(0,\frac{x_1}{\sqrt{N(x_1)}}) \in \Spin(\octo,N)$ and $\xi(h_2)=:g_2\in\F$. 
By $(\ref{eq:chi9})$, $g_2$ sends the spinor part
$x_1 \oplus x_2$ of $g_1\cdot A$ to $x_1'\oplus x_2'$ where $x_1'=\sqrt{N(x_1)} \in \C$ and 
$x_2'=\frac{1}{\sqrt{N(x_1)}}x_1 x_2$. The matrix $(g_2g_1\cdot A)^2$ has the same form as
$(\ref{matrixsquare})$ with $x_1$ and $x_2$ substituted by $x_1'$ and $x_2'$. It is still a zero matrix and its
$(2,3)$-position $0=x_1'\bar{x}_2'$ implies $x_2'=0$ ($x_1'$ is a nonzero complex number).
The other positions of this matrix imply $0=r_3^2+N(x'_2)$, so $r_3=0$, and $r_1^2+N(x_1')=r_1^2+(x_1')^2=0$,
so 

$$ 
g_2 g_1\cdot A=
         \begin{pmatrix}
               \pm \imath w &  w & 0 \\
                  w & \mp \imath w & 0 \\
               0 & 0& 0
         \end{pmatrix}
$$
for some  $0\neq w\in \C$.

      The case $N(x_1)=0$, $N(x_2)\neq 0$ leads in a similar way to a matrix of the form
      \(
         \begin{smatrix}
               \pm \imath w & 0 & w  \\
               0 & 0& 0\\
               w & 0 &  \mp \imath w \\
         \end{smatrix}
      \), $0\neq w\in \C$,
      which can be transformed by the orthogonal matrix
      \(
      \begin{smatrix}
         1 & 0 & 0 \\
         0 & 0 & 1 \\
         0 & 1 & 0
      \end{smatrix}
      \)
      to the previous one. 
One can get rid of the sign ambiguity with
      \(\begin{smatrix} 0&1&0\\1&0&0\\0&0&1\end{smatrix}\)
and the matrix
      \(
       \begin{smatrix}
               \imath w &  w & 0 \\
                  w & - \imath w & 0 \\
               0 & 0& 0
         \end{smatrix}
      \)
can be transformed to the canonical form \eqref{eq:canform} by conjugating by the orthogonal matrix
      \[
        \begin{pmatrix}
            \frac {1}{\sqrt{w}}&0&{\frac {-\imath\sqrt{
            1-w}}{\sqrt{w}}}\\{\frac {-\imath \left( 1-w \right) }{
            \sqrt{w}}}&\sqrt{w}&-{\frac {\sqrt{1-w}}{\sqrt{w}}}
            \\\imath\sqrt{1-w}&\sqrt{1-w}&1
         \end{pmatrix}.
      \]

So, $g_3 g_2 g_1\cdot A$ has the canonical form $(\ref{eq:canform})$, 
where $g_3$ is some element in the image of the embedding 
$O(3,\C)\hookrightarrow \F$ defined in the section \ref{liealgebra}.

If $A$ has isotropic but nonzero vector part, then the preceding lemma implies that the 
topleft element of $A$ is $0$. Using the lemma  \ref{lem:vect_trans} we can find an 
element $g'\in\xi(\Spin(\V, N'))\leq \F$ such that 
\( g'\cdot A=
        \begin{smatrix}
            0 &\bar{x}_1 & \bar{x}_2 \\
            x_1 & \imath w & w\\
            x_2 & w & -\imath w
        \end{smatrix}
\)
where $w\neq0$. 
%
%
Conjugation by $\begin{smatrix}
                                          0 & 1 &0 \\
                                          1 & 0 & 0\\
                                          0 & 0 & 1
                                          \end{smatrix}$
leads to a matrix whose top left element is $\imath w\neq 0$. 
By the previous lemma, such a matrix has nonisotropic vector part
and we have reduced this case to the already solved one.

Finally, suppose that $A$ has zero vector part,
      \(
       A=\begin{smatrix}
        0  & \bar{x}_1&\bar{x}_2\\
        x_1 & 0 &0\\
        x_2 &0&0       
       \end{smatrix}.
      \) 
This matrix is nonzero by definition.
If $x_2\neq 0$, then the action of 
$\begin{smatrix}
  0 & 1 &0 \\
  1 & 0 & 0\\
  0 & 0 & 1
\end{smatrix}$
transforms it to a matrix with nonzero vector part. The case $x_1\neq0$ is treated similarly.
\end{proof}

\begin{remark}
   We see from the proof that in order to prove transitivity of $\F$ on $\plane$, it is sufficient to consider only discrete subgroup of $\mathrm{O}(3,\C)$ isomorphic to $\mathrm{S}_3$ -- a permutation group on three letters. This is a manifestation of the triality principle.
\end{remark}

Now we prove that the cone $\affOP$ over $\plane$ is a smooth manifold.
\begin{proposition}
   The space $\affOP$ is a smooth manifold of dimension $32$.
\end{proposition}
\begin{proof}
Let as define the smooth map $f:\Herm(3,\octo)_0 \to \Herm(3,\octo)_0$ by  $f(A):= A^2$.
   We use the implicit function theorem to show that $\affOP=f^{-1}(0)\setminus\{0\}$ is a smooth manifold.
The differential of $f$ at $A$ is easily proved to be $B \mapsto 2 A\circ B$. 
We already know that $\F$ acts transitively on $f^{-1}(0)\setminus\{0\}=\affOP$ and so we have
   \[\dim \ker (B \mapsto A\circ B) = \dim \ker \big(B \mapsto g\cdot (A\circ (g^{-1}\cdot B))\big) = 
\dim \ker \big(B \mapsto (g\cdot A) \circ B\big) \]
   for any $g\in \F$. So, the differential $df$ of $f$ has constant rank on the set 
$f^{-1}(0)\setminus \{0\}$ and $\affOP$ is a smooth manifold.

   The kernel of the differential of $f$ at the canonical point \eqref{eq:canform} equals
   \[\left\{ \begin{pmatrix}
              \imath \Re (x_1) &           x_1          &     x_2 \\

                 \bar{x}_1     &     -\imath\Re (x_1)   & -\imath x_2 \\
                 \bar{x}_2     &     -\imath \bar{x}_2  & 2\Re(x_1)
             \end{pmatrix} \middle |\, x_1,x_2 \in \octo\right\}
   \]
   and is isomorphic to the tangent space of $\affOP$ at that point.
\end{proof}

\subsection{The real case}

By choosing an appropriate involution on $\mathcal{J}_3(\octo_\C)$ we get a model for $\F^{(-20)}/\P$ -- i.e. the conformal ifinity of the Einstein space $\mathbb{O}H^2$. According to Yokota \cite{Yokota} the following real subalgebra of $\mathcal{J}_3(\octo_\C)$  
\[\left\{ A \in \mathcal{J}_3(\octo_\C) : I_1\overline{A}^T I_1 = A, I_1 = \begin{smatrix}
                                   -1 & 0 & 0\\
				    0 & 1 & 0\\
				    0 & 0 & 1
                                  \end{smatrix}
 \right\} = 
  \left\{
\begin{pmatrix}
 r_1 & x_1 & x_2 \\
 - \overline{x}_1 & r_2 & x_3\\
 - \overline{x}_2 & \overline{x}_3 & r_3
\end{pmatrix} : x_i \in \octo_\R, r_i\in \R
\right\}
\]
has $\F^{(-20)}$ as its automorphism group. By restricting the map $\kappa$ to $\R\oplus \octo_\R$ we get presentation of $\Spin(9,\R)$ and the restriction of our representation $\xi$ maps $\Spin(9,\R)$ into $\F^{(-20)}$. Instead of $\mathrm{O}(3,\C)$ we have an orthogonal group of indefinite signature $\mathrm{O}(1,2,\R)$.

The model of $\F^{(-20)}/\P$ is given by the same equations as in the complex case. Since there are no isotropic vectors, the proof of transitivity is now much simpler. By transitivity of $\mathrm{SO}(9,\R)$ on spheres we can map any element of our model to a matrix of the form 
$\begin{smatrix}
 -2t & x_1 & x_2\\
  -\overline{x}_1 & t+s & 0 \\
 -\overline{x}_2 & 0 & t-s                                                                                                                                                                                                                                                                                                                                                                                                                                                                                                          
\end{smatrix}$. The square of this matrix has to be zero by definition which for diagonal elements gives three equations that yield  easily $t^2-s^2 = 0$. The case $t=-s$ leads to $x_1=0$ and can be reduced to the case of $t=s$ by conjugation with $\begin{smatrix} 1 & 0 &0\\0&0&1\\0&1&0\end{smatrix}$. 

 The case $t=s$ gives $x_2=0$ and we can easily find an action of $\Spin(8,\R)$ to map $x_1$ to a real number which gives us a matrix in the final form 
$\begin{smatrix}
  -r & x & 0\\
   -x & r&0\\
   0 & 0 & 0
 \end{smatrix}$, where all the entries are real and $r^2 = x^2$. If $r=x$, we can assume that $x$ is positive and then conjugation by the matrix $\frac{1}{2}\sqrt{\frac{1}{x}} \begin{smatrix} x+1 & 1-x & 0 \\ 1-x & x+1 &0 \\ 0&0&2\sqrt{x}\end{smatrix}$ gives the canonical form $\begin{smatrix} -1 & 1 &0 \\-1&1&0\\0&0&0\end{smatrix}$. If $r=-x$, we assume $x$ to be negative and use conjugation by $\frac{1}{2}\sqrt{\frac{-1}{x}} \begin{smatrix} x-1 & x+1 & 0 \\ -(x+1) & 1-x &0 \\ 0&0&2\sqrt{-x}\end{smatrix}$.

We see that the automorphism groups again acts transitively even on the cone over the projectivization. In contrast to the complex case however we need only subgroup of $\mathrm{O}(1,2,\R)$ isomorphic to $\mathbb{Z}_2$.

\section[Parabolic stabilizer]{Description of the stabilizer of the $\F$ action}

In this section we will identify the stabilizer of $\plane$ as a concrete parabolic subgroup of $\F$. 
\begin{lemma}
There exists up to isomorphism only one irreducible representation $\varrho$ of the group $\F$ such that
\[
1 < \dim_\C\,\varrho\leq 26.
\]
The highest weight of this representation is $\varpi_4=\epsilon_1$.
\end{lemma}

\begin{proof}
Let $\lambda, \mu \in \mathfrak{h}^*$  be two integral dominant weights, $\mu\neq 0$. By a direct application of the 
Weyl dimensional formula (see Goodman, Wallach~\cite{GoodmanWallach}), we obtain that
$\dim \varrho_{\lambda+\mu}>\dim \varrho_\lambda$.
Using the program LiE \cite{LiE}, we get $\dim\rho_{\varpi_1}=52$, $\dim\varrho_{\varpi_2}=1274$, $\dim\varrho_{\varpi_3}=273$ and $\dim\varrho_{\varpi_4}=26$. 
By the previous inequality, we see that there is only one irreducible $26$-dimensional representations of the Lie algebra $\mathfrak{f}_4$.
\end{proof}

Since $\dim\mathcal{J}_0=26$ and all finite dimensional representation of the simple Lie group $\F$ are completely reducible, we obtain immediately the following.
\begin{proposition}
\label{defrep}
The restriction to the defining representation of $\F$ on $\mathcal{J}_0=\Herm(3,\octo)_0$ 
is isomorphic to the $26$-dimensional irreducible representation $\varrho_{\epsilon_1}$.
\end{proposition}

It is clear from definition that $\plane$ is a projective variety. According to Humphreys \cite{Humphreys} this implies that the stabilizer group of any point is a parabolic subgroup of $\F$. Since any parabolic subgroup contains Borel subgroup, it follows that the points of the variety are lines spanned by highest weight vectors.

For a fixed choice of the Cartan subelgebra $\mathfrak{h}$ and simple roots $\Delta$ there is a $1-1$ correspondence between isomorphism classes of parabolic subalgebras $\mathfrak{p}\subseteq\mathfrak{g}$ 
and subsets $\Sigma \subseteq \Delta$ of the set $\Delta$ of simple roots described e.g. in \cite[chapter 3]{CapSlovak}.
We will denote the parabolic subalgebra corresponding to $\Sigma=\{\alpha_i\}$ by $\mathfrak{p}_i$.

Because the highest weight of $\mathcal{J}_0$ is $\epsilon_1$, the folloving theorem follows directly from \cite[Theorem 3.2.5]{CapSlovak}. Its proof is not difficult -- it is based on the fact that for each $X\in\mathfrak{g}_\alpha$ one can find $Y\in\mathfrak{g}_{-\alpha}$ such that $[Y,X] = H_\alpha$, where $H_\alpha(\lambda) = \langle \lambda,\alpha \rangle$ and the fact that the set of weights is invariant under the action of Weyl group.

\begin{theorem}
Let $P$ be the stabilizer of a point $p\in\plane$ with respect to the action of the group $\F$. 
Then the Lie algebra $\mathfrak{p}$ of the group $P$ is isomorphic to $\mathfrak{p}_4$.
\end{theorem}

\begin{remark}
We see that $\affOP$ is the $\F$-orbit of the highest weight vector in $\mathcal{J}_0$. Points in $\affOP$ are exactly
all possible highest weight vectors for this representation, corresponding to different choices of $\mathfrak{h}$ and
$\Phi^+$. The real case can be treated in similar manner with analogous results. See \cite{CapSlovak} for details.
\end{remark}

\begin{remark}
    From the computation of the harmonic curvature (as done for example in \cite{Krysl}, also see \cite{CapSlovak}) one can prove that the homogeneous space doesn't admit curved deformations in the sense of regular normal Cartan geometries. However, if one relaxes the regularity condition there are some deformations of this structure \cite{Armstrong}.
\end{remark}


\begin{thebibliography}{50}
    \bibitem{Armstrong} Armstrong, S., Biquard O.: \textit{Einstein metrics with anisotropic boundary behaviour}, arXiv:0901.1051v1 [math.DG].
 \bibitem{Atiyah} Atiyah, M., Berndt, J.: \textit{Projective planes, Severi varieties and spheres}, Surveys in Differential Geometry, International Press of Boston Inc, 2003
    \bibitem{Baez} Baez, J. C.: \textit{The Octonions}, Bull. of Am. Math. Soc. \textbf{39} Nr. 2, 145--205, 2001.

    \bibitem{Biquard} Biquard, O.: \textit{Asymptotically symmetric {E}instein metrics}, SMF/AMS Texts and Monographs, American Mathematical Society, 2006.

   \bibitem{Bourbaki} Bourbaki, N.: \textit{Lie groups and Lie algebras. Chapters 4--6}, Elements of Mathematics, Springer-Verlag, Berlin, 2002. 

    \bibitem{CapSlovak} \v{C}ap, A., Slov\'ak, J.: \textit{Parabolic Geometries: Background and general theory}, Mathematical Surveys and Monographs, AMS Bookstore, 2009

    \bibitem{Dray} Dray, T., Manogue, C.A.: \textit{Octonionic Cayley Spinors and $\mathrm{E}_6$}, Comment. Math. Univ. Carolin. 51, 193–207 (2010)

    \bibitem{Friedrich} Friedrich, T.: \textit{Weak Spin(9)-Structures on 16-dimensional Riemannian Manifolds}, Asian Journal of Mathematics 5 (2001), pp. 129-160, arXiv:math/9912112v1 [math.DG]



    \bibitem{GoodmanWallach} Goodman, R., Wallach, N.R.: \textit{Representations and invariants of the classical groups}, Cambridge University Press, Cambridge, 1998.

    \bibitem{Harvey} Harvey, F. R.: \textit{Spinors and Calibration}, Academic Press, San Diego, 1990.

    \bibitem{Humphreys} Humphreys, J. E.: \emph{Linear algebraic groups}, Graduate Texts in Mathematics, No. 21. Springer-Verlag, 1975.

    \bibitem{Jacobson} Jacobson, N.: \emph{Structure and Representations of Jordan Algebras},
      AMS Bookstore, 2008

    \bibitem{JNW} Jordan, P., von Neumann, J., Wigner, E., \textit{On an algebraic generalization of the quantum mechanical formalism},
Ann. of Math. (2) 35 (1934), no. 1, 29--64. 

  \bibitem{K_clas1} Kr\'ysl, S.: \textit{Classification of $1^{st}$ order symplectic spinor operators in contact projective geometries}, Diff. Geom. Appl., Vol. 26, Issue 3, Elsvier, 2008.
  \bibitem{K_clasp} Kr\'ysl, S.: \textit{Classification of $\mathfrak{p}$-homomorphisms between higher symplectic spinors}, Rend. Circ. Mat. di Palermo, 2006.
    \bibitem{Krysl} Kr\'ysl, S.: \textit{BGG Diagrams for Contact Graded Odd Dimensional Orthogonal Geometries}, Acta Universitatis Carolinae Mathematica et Physica Vol. 45 No.1, Prague, 2004.


    \bibitem{Landsberg} Landsberg, J. M., Manivel, L.: \textit{On the projective geometry of rational homogenous varieties}, Comment. Math. Helv. \textbf{78} No. 1, 65--100, 2003.

   \bibitem{Landsberg2} Landsberg, J. M., Manivel, L.: \textit{The projective geometry of Freudenthal's magic square},
    Journal of Algebra  239, 477-512 (2001)


    \bibitem{LiE} van Leeuwen, M. A. A.,  Cohen, A. M. and  Lisser, B.: \textit{LiE, A Package for Lie Group Computations}, Computer Algebra Nederland, Amsterdam, 1992, available at \url{http://www-math.univ-poitiers.fr/~maavl/LiE/}

    \bibitem{Moufang} Moufang, R.: \textit{Alternativk\"orper und der Satz vom vollst\"andigen Vierseit},
    Abhandl. Math. Univ. Hamburg 9 (1933), 207-222.


        \bibitem{Springer} Springer, T. A.,  Veldkamp, F. D.: \textit{{O}ctonions, {J}ordan algebras and exceptional groups}, Springer Monographs in Mathematics. Springer-Verlag, Berlin, 2000. 

  \bibitem{Yokota} Yokota, I.: \textit{Exceptional Lie groups},
  arXiv.org:0902.0431 [math.DG], 2009.

\end{thebibliography}
\end{document}